\documentclass[a4paper,11pt]{article}
\usepackage{amsthm,amsmath,amssymb,amsfonts,mathrsfs,hyperref,microtype}
\usepackage{graphicx,color}
\usepackage{eufrak}

\theoremstyle{definition}
\newtheorem{dfn}{Definition}[section]
  
\theoremstyle{plain}
\newtheorem{thm}{Theorem}[section]
\newtheorem{pro}{Proposition}[section]
\newtheorem{cor}{Corollary}[section]
\newtheorem{lma}{Lemma}[section]
\theoremstyle{definition}
\newtheorem{rem}{Remark}[section]
\newtheorem{exa}{Example}[section]

\newcommand{\R}{\mathbb{R}}
\newcommand{\E}{\mathbb{E}}
\renewcommand{\P}{\mathbb{P}}

\newcommand{\HH}{\mathcal{H}}

\newcommand{\la}{\langle}
\newcommand{\ra}{\rangle}

\renewcommand{\d}{\mathrm{d}}
\newcommand{\law}{\stackrel{\text{law}}{=}}

\numberwithin{equation}{section}

\begin{document}
\title{\textbf{Sufficient and Necessary Conditions for Limit Theorems for Quadratic Variations of Gaussian Sequences }}
\author{Lauri Viitasaari}

\renewcommand{\thefootnote}{\fnsymbol{footnote}}

\author{Lauri Viitasaari\footnotemark[1] \footnotemark[2]}

\footnotetext[1]{Department of Mathematics and System Analysis, Aalto University School of Science, Helsinki P.O. Box 11100, FIN-00076 Aalto,  Finland}

\footnotetext[2]{Department of Mathematics, Saarland University, Saarbr\"ucken\\
Postfach 151150, D-66041 Saarbr\"ucken, Germany}

\maketitle

\begin{abstract}
Quadratic variations of Gaussian processes play important role in both stochastic analysis and in applications such as estimation of model parameters, and for this reason the topic has been extensively studied in the literature. In this article we study the problem for general Gaussian processes and we provide sufficient and necessary conditions for different types of convergence which include convergence in probability, almost sure convergence, $L^p$-convergence as well as convergence in law. Furthermore, we study general Gaussian vectors from which different interesting cases including first or second order quadratic variations can be studied by appropriate choice of the underlying vector. Finally, we provide a practical and simple approach to attack the problem which simplifies the existing methodology considerably. 

\medskip


\small skip

\end{abstract}

\tableofcontents


\section{Introduction}

Quadratic variation of a stochastic process $X$ plays an important role in different applications. For example, the concept is important is one is interested to develop stochastic calculus with respect to $X$. Furthermore, quadratic variations can be used to build estimators for the model parameters such as self-similarity index or parameter describing long range dependence which have important applications in all fields of science such as hydrology, chemistry, physics, and finance to simply name a few. For such applications one is interested to study the convergence of the quadratic variation. Furthermore, a wanted feature is to obtain a central limit theorem which allows one to apply statistical tools developed for normal random variables.

For Gaussian processes the study of quadratic variation goes back to L\'evy who studied standard Brownian motion and showed the almost sure convergence
$$
\lim_{n\to\infty}\sum_{k=1}^{2^n} \left[W_{\frac{k}{2^n}}-W_{\frac{k-1}{2^n}}\right]^2 =1.
$$
Later this result was extended to cover more general Gaussian processes in Baxter \cite{Bax} and in Gladyshev \cite{glady} for uniformly divided partitions. General subdivisions were studied in Dudley \cite{dud} and Klein \& Gine \cite{klei-gine} where the optimal condition $o\left(\frac{1}{\log n}\right)$ for the mesh of the partition was obtained in order to obtain almost sure convergence. It is also known that for the standard Brownian motion the condition $o\left(\frac{1}{\log n}\right)$ is not only sufficient but also necessary. For details on this topic see De La Vega \cite{vega} for construction, and \cite{le-er} for recent results.
Functional central limit theorem for general class of Gaussian processes were studied in Perrin \cite{perrin}. More recently Kubilius and Melichov \cite{kub-mel} defined a modified Gladyshev's estimator and the authors also studied the rate of convergence. Norvai\^sa \cite{norvaisa} have extended Gladyshev's theorem to a more general class of Gaussian processes. Finally, we can mention a paper by Malukas \cite{malukas} who extended the results of Norvai\^sa to irregular partitions, and derived sufficient conditions for the mesh in order to obtain almost sure convergence. 

The case of fractional Brownian motion with Hurst index $H\in(0,1)$ were studied in details by Gyons and Leons \cite{gyo-leo} where the authors showed that appropriately scaled first order quadratic variation (that is, the one based on differences $X_{t_k} - X_{t_{k-1}}$) converges to a Gaussian limit only if $H<\frac{3}{4}$. To overcome this problem, a generalisations of quadratic variations were used in \cite{I-L}, \cite{benassi-et-al}, \cite{co}, and \cite{cohen-et-al}. The most commonly used generalisation is second order quadratic variations based on differences $X_{t_{k+1}}-2X_{t_k}+X_{t_{k-1}}$ which was studied in details in a series of papers by Begyn \cite{begyn-ejp,begyn-spa,begyn-ber} with applications to fractional Brownian sheet and time-space deformed fractional Brownian motion. In particular, in \cite{begyn-ejp} the sufficient condition for almost sure convergence was studied with non-uniform partitions. The central limit theorem and its functional version were studied in \cite{begyn-ber} and \cite{begyn-spa} with respect to a standard uniform divided partitions. Furthermore, the authors in papers \cite{Breu-major,taqqu} have studied more general variations assuming that the underlying Gaussian process have stationary increments. For another generalisation, the localised quadratic variations were introduced in \cite{benassi-et-al2} in order to estimate the Hurst function of multifractional Gaussian process. These results have been generalised in \cite{coeurjolly,lacaux}. 

The fractional Brownian motion has received a lot of attention in modelling as a (relatively) simple generalisation of a standard Brownian motion. However, for some applications the assumption of stationary increments is an unwanted feature. For this reason there is a need for extensions of fractional Brownian motion. Recent such generalisations are sub-fractional Brownian motion depending on one parameter $H\in(0,1)$ introduced by Bojdecki et al. \cite{bojdecki-et-al}, and bifractional Brownian motion depending on two parameters $H\in(0,1)$ and $K\in(0,1]$ (the case $K=1$ corresponding to the fractional Brownian motion) introduced by Houdr\'e and Villa \cite{h-v}, and later studied in more details by Russo and Tudor \cite{ru-tu}. Furthermore, bifractional Brownian motion was extended for values $H\in(0,1)$, $K\in(1,2)$ satisfying $HK\in(0,1)$ in \cite{bar-seb}.

Recently there has also been interest in general Hermite variations which are partially motivated by the contributing paper by Breuer \& Major \cite{Breu-major}. For results related to fractional Brownian motion, we refer to \cite{bre-nou,nou-et-al,nou-rev}. The integrals of fractional Brownian motion were studied in \cite{cor-et-al}. Moreover, fractional Brownian sheet have received attention at least in papers \cite{rev,rev-et-al,pak-rev}. Furthermore, generalisations are studied in \cite{pakkanen} and \cite{barn-et-al}. The mentioned papers studying general Hermite variations are not focused on practical importance of quadratic variations. Furthermore, the proofs are based on now well-developed Malliavin calculus.

This paper aims to take a practical, instructive, and a general approach to quadratic variations. Firstly, our aim is to provide easy to check conditions for practitioners which still covers many interesting cases. Secondly, with our simplified approach we are able to provide intuitively clear explanations such as the discussion in subsection \ref{subsec:a-qv} rather than get lost into technical details. Finally, to obtain the generality we study sequences of general $n$-dimensional Gaussian vectors $Y^n=\left(Y_1^{(n)}, \ldots, Y_n^{(n)}\right)$, where each component $Y^{(n)}_k$ may depend on $n$, and we study the asymptotic behaviour of the vector $Y^n$ or its quadratic variation defined as the limit
$$
\lim_{n\to \infty} \sum_{k=1}^n \left[Y^{(n)}_k\right]^2
$$
provided it exists in some sense. As such, different cases such as first or second order quadratic variations can be obtained by choosing the vectors $Y^n$ suitably, and this fact will be illustrated in the present paper.  

We begin by providing sufficient and necessary conditions for the convergence in probability which, applied to some quadratic functional of a given process, can be used to construct consistent estimators of the model parameters. Furthermore, we show that in this case the convergence holds also in $L^p$ for any $p\geq 1$. We will also apply the well-known Gaussian concentration inequality for Hilbert-valued Gaussian random variables which provides a simple condition that guarantees the almost sure convergence. This condition is applied to quadratic variations of Gaussian processes with non-uniform partitions for which we obtain sufficient conditions for the convergence. More importantly, this result is shown to hold for many cases of interest and in the particular case of standard Brownian motion, this condition corresponds to the known sufficient and necessary condition. Compared to the existing literature, in many of the mentioned studies the almost sure convergence is obtained by the use of Hanson and Wright inequality \cite{hanson-wright} together with some technical computations. In this paper we show how these results follow easily from Gaussian concentration phenomena.

We will also study central limit theorem in our general setting. We begin by providing sufficient and necessary conditions under which appropriately scaled quadratic variation converges to a Gaussian limit. To obtain this result we apply a powerful fourth moment theorem proved by Nualart and Peccati \cite{nu-pe} which, thanks to the recent results by Sottinen and the current author \cite{sot-vii}, can essentially be applied always. We will also show how a version of Lindeberg's central limit theorem for this case follows easily. Finally, we will use some well-known matrix norm relations to obtain surprisingly simple way to obtain a convergence towards a normal random variable. More remarkably, it seems that this simple condition is essentially the one used in many of the mentioned studies while the result is derived case by case. We will also provide a Berry-Esseen type bound that holds in our general setting which,  to the best of our knowledge, is not present in the literature excluding some very special cases. Furthermore, our approach does not require the knowledge of Malliavin calculus and should be applicable to anyone with some background on linear algebra and Gaussian vectors.

To summarise, in this paper we give sufficient and necessary conditions for the convergence of guadratic variations of general Gaussian vectors which can be used to reproduce and generalise the existing results. Furthermore, we give easily checked sufficient conditions how one can obtain the wanted convergence results. As such, with our approach we are able to generalise the existing results as well as simplify the proofs considerably by relying on different techniques. At the best, the methods and results of this paper would provide new tools to attack the problem under consideration while classically the problem is studied by relying on Hanson and Wright inequality together with Lindeberg's central limit theorem. 

The rest of the paper is organised as follows. In section \ref{sec:conv} we study general Gaussian vectors and provide our main results. In section \ref{sec:qv} we illustrate how our results can be used to study quadratic variations. We will consider non-uniform sequences and generalise some of the existing results. The main emphasis is on first order quadratic variations which is more closely related to stochastic calculus while we also illustrate how second order quadratic variations can be studied with our approach. We end section \ref{sec:qv} with a discussion on general quadratic variations which is close in spirit with the work by Istas and Lang \cite{I-L}. Finally, section \ref{sec:ex} is devoted to examples.

\section{Convergence of Gaussian sequences}
\label{sec:conv}
\subsection{Notation and first results}
Let $Y^n=\left(Y_1^{(n)},Y_2^{(n)},\ldots, Y_n^{(n)}\right)$ be a zero mean Gaussian vector, where each $Y_k^{(n)}$ possibly depends on $n$. 
We consider properties of sequences $Y^n$ as $n$ tends to infinity. Throughout the paper we will also use Landau notation, i.e. for a sequences $a_n$ and $b_n$ we denote;
\begin{itemize}
\item $a_n = O(b_n)$ if $\sup_{n\geq 1}\frac{|a_n|}{|b_n|} < \infty$,
\item $a_n = o\left(b_n\right)$ if $\lim_{n\to\infty}\frac{|a_n|}{|b_n|} = 0 $. 
\end{itemize} 
We also denote $a_n \sim b_n$ as $n\to \infty$ if $\lim_{n\to \infty}\frac{a_n}{b_n} \to 1$. 

We begin with the following definition which is a discrete analogy to the similar concepts introduces in \cite{ru-va2}.
\begin{dfn}
\begin{enumerate}
\item
We say that the sequence $Y=(Y^n)_{n=1}^\infty$ has a quadratic variation $\la Y\ra$ if the random variable $\la Y\ra:=\lim_{n\to \infty} \sum_{k=1}^n \left(Y_k^{(n)}\right)^2$ exists as a limit in probability.
\item
The energy of a sequence $Y=(Y^n)_{n=1}^\infty$ is defined as the limit 
$$
\varepsilon(Y):=\lim_{n\to \infty} \sum_{k=1}^n \E\left(Y_k^{(n)}\right)^2
$$
provided the limit exists.
\item
We say that $Y$ has $2$-planar variation defined as the limit 
$$
\Upsilon(Y):=\lim_{n,m\to \infty} \sum_{k=1}^n \sum_{j=1}^m\left[\E\left[Y_k^{(n)}Y_j^{(m)}\right]\right]^2
$$
provided the limit exists.
\end{enumerate}
\end{dfn}
We will also denote 
\begin{equation}
\label{eq:V_n}
V_n = \sum_{k=1}^n \left[\left(Y_k^{(n)}\right)^2 - \E\left(Y_k^{(n)}\right)^2\right]
\end{equation}
for the centered quadratic variation.
\begin{exa}
Let $X=(X_t)_{t\in[0,T]}$ be a centred Gaussian process, $t_k^n = \frac{k}{n}$ and let $\Delta_k X$ denote some difference of the Gaussian process $X$. 
By setting $Y^{(n)}_k = \frac{\Delta_k X}{n\E (\Delta_k X)^2}$ we have $\E \left(Y_k^{(n)}\right)^2 = n^{-1}$ and
$$
V_n = \frac{1}{n} \sum_{k=1}^n \frac{\Delta_k X}{\E (\Delta_k X)^2} - 1.
$$
Hence by setting $\Delta_k X = X_{t_k^n} - X_{t^n_{k-1}}$ we obtain the first order quadratic variation with respect to uniform partition. Similarly, by setting 
$\Delta_k X = X_{t^n_{k+1}} - 2X_{t^n_k} + X_{t^n_{k-1}}$ we obtain the second order quadratic variation with respect to uniform partition. 
\end{exa}
The Euclidean distance of the vector $Y^n=(Y_1^{(n)},Y_2^{(n)},\ldots, Y_n^{(n)})$ is given by
$$
\Vert Y^n\Vert_2 = \sqrt{\sum_{k=1}^n \left(Y_k^{(n)}\right)^{2}}.
$$
Note that the norm $\Vert \cdot \Vert_2$ also depends on the dimension $n$ which we will omit on the notation.
We will denote by $\Gamma^{(n)}$ the covariance matrix of the vector $Y^n$, i.e. the $n\times n$-matrix with elements 
\begin{equation}
\Gamma_{jk}^{(n)} = \E\left(Y_j^{(n)}Y_k^{(n)}\right).
\end{equation}
Note now that with this notation the energy of a process $Y$ is simply the limit of the trace of the matrix $\Gamma$, i.e. $\varepsilon(Y) = \lim_{n\to\infty}\text{trace}(\Gamma^n)$. Similarly, $Y^n$ has a quadratic variation if the limit $\Vert Y^n\Vert_2^2$ converges as $n$ tends to infinity. Recall next that the Frobenius norm of a matrix $\Gamma^{(n)} = (\Gamma^{(n)}_{ij})_{i,j=1,\ldots,n}$ is given by
$$
\Vert \Gamma^{(n)} \Vert_F = \sqrt{\sum_{i,j=1}^n \left(\Gamma_{ij}^{(n)}\right)^2}.
$$
Hence we have
\begin{equation}
\label{eq:frob_limit}
\lim_{n\to\infty} \Vert \Gamma^{(n)} \Vert_F^2 = \lim_{n\to\infty} \sum_{k=1}^n \sum_{j=1}^n \left[\E\left[Y_k^{(n)}Y_j^{(n)}\right]\right]^2.
\end{equation}
We will later show that in interesting cases we also have
$$
\lim_{n,m\to\infty} \sum_{k=1}^n \sum_{j=1}^m\left[\E\left[Y_k^{(n)}Y_j^{(m)}\right]\right]^2 = \lim_{n\to\infty} \sum_{k=1}^n \sum_{j=1}^n \left[\E\left[Y_k^{(n)}Y_j^{(n)}\right]\right]^2
$$
which, in view of \eqref{eq:frob_limit}, shows that the $2$-planar variation $\Upsilon(Y)$ is given by
$$
\Upsilon(Y) = \lim_{n\to \infty} \Vert \Gamma^{(n)} \Vert_F^2.
$$
The following first result concerns convergence in probability. The proof follows essentially the arguments presented in \cite{ru-va2} and is based on cumulant formulas for Gaussian random variables. The main difference is that we also prove the convergence in $L^p$ for any $p\geq 1$ which have some important consequences in stochastic analysis (to be detailed in a forthcoming paper) while in \cite{ru-va2} the authors considered only $L^2$-convergence. The proof follows the ideas presented in \cite{ru-va2} but we will present the key points for the sake of completeness.
\begin{thm}
\label{thm:QV-ex}
Let $Y=(Y^n)_{n=1}^\infty$ be a sequence of Gaussian vectors with finite energy. Then quadratic variation exists as a limit in probability for every $t\geq 0$ if and only if the sequence $(Y^n)_{n=1}^\infty$ has 2-planar variation. In this case, the convergence holds also in $L^p$ for any number $p\geq 1$.
\end{thm}
\begin{proof}
Denote 
$$
Z_n = \Vert Y^n \Vert_{2}^{2}=\sum_{k=1}^n \left(Y_k^{(n)}\right)^2.
$$
We start by showing that 
$$
Z^p_n = \left[\sum_{k=1}^n \left(Y^{(n)}_k\right)^2\right]^p
$$
is uniformly integrable for any $p\geq 1$.
By Minkowski's inequality for measures we get
\begin{equation*}
\begin{split}
&\left[\E \left(\sum_{k=1}^n\left(Y_k^{(n)}\right)^2\right)^p\right]^{\frac{1}{p}}\\
&\leq \sum_{k=1}^n \left[\E\left(Y_k^{(n)}\right)^{2p}\right]^{\frac{1}{p}}\\
&= C_p \sum_{\pi_n} \E\left(Y_k^{(n)}\right)^2
\end{split}
\end{equation*}
by the fact that $Y_k$ is Gaussian. Now this upper bound converges to $\varepsilon(Y)$, and consequently the quantity $\left(\sum_{k=1}^n \left(Y^{(n)}_k\right)^2\right)^p$ is uniformly integrable for any $p\geq 1$. 
Now we have
$$
\E(Z_n-Z_m)^2 = \E Z_n^2 + \E Z_m^2 - 2 \E(Z_nZ_m),
$$
where
$$
\E(Z_nZ_m)=\sum_{k=1}^n\sum_{j=1}^m \E\left[\left(Y^{(n)}_k\right)^2\left(Y^{(m)}_j\right)^2\right]
$$
Here we have 
\begin{equation}
\label{eq:basic_cumulant}
\E\left[\left(Y^{(n)}_k\right)^2\left(Y^{(m)}_j\right)^2\right] = 2\left[\E\left(Y^{(n)}_kY^{(m)}_j\right)\right]^2 + \E\left(Y^{(n)}_k\right)^2\E\left(Y^{(m)}_j\right)^2.
\end{equation} 
Hence we have
\begin{equation}
\sum_{k=1}^n\sum_{j=1}^m \E\left[\left(Y^{(n)}_k\right)^2\left(Y^{(m)}_j\right)^2\right] = 2\sum_{k=1}^n\sum_{j=1}^m \left[\E\left(Y^{(n)}_kY^{(m)}_j\right)\right]^2 + \sum_{k=1}^n \E \left(Y^{(n)}_k\right)^2 \times \sum_{k=1}^n \E \left(Y^{(n)}_k\right)^2.
\end{equation}
Consequently, 
\begin{equation}
\label{eq:key_relation}
\E(Z_nZ_m) = 2\sum_{k=1}^n\sum_{j=1}^m \left[\E\left(Y^{(n)}_kY^{(m)}_j\right)\right]^2 + \sum_{k=1}^n \E (Y^{(n)}_k)^2 \times \sum_{k=1}^n \E (Y^{(n)}_k)^2,
\end{equation}
where
$$
\sum_{k=1}^n \E (Y^{(n)}_k)^2 \times \sum_{k=1}^n \E (Y^{(n)}_k)^2 \rightarrow \varepsilon(Y)^2
$$
by the fact that $Y$ has finite energy. Recall also that 2-planar variation is given by
$$
\Upsilon(Y) =\lim_{n,m\to\infty} \sum_{k=1}^n\sum_{j=1}^m \left[\E\left(Y^{(n)}_kY^{(m)}_j\right)\right]^2.
$$
Hence relation \eqref{eq:key_relation} implies the result. Indeed, assuming that $Z_n$ converges in probability, then uniform integrability implies that $\la Y\ra \in L^p$ and the convergence holds also in $L^p$. Hence $\E(Z_nZ_m)$ converges, and \eqref{eq:key_relation} implies that 2-planar variation exists. Conversely, if 2-planar variation exists, then \eqref{eq:key_relation} implies that $\E(Z_nZ_m)$ converges to the same limit as $\E(Z_n^2)$ which concludes the proof. 
\end{proof}
\begin{rem}{\rm
It is straightforward to check that the $L^p$-convergence takes place also in a continuous setting of Russo and Vallois \cite{ru-va2}.
}\end{rem}
The following theorem gives condition when the quadratic variation is deterministic. It seems that this is indeed true in many cases of interest.
\begin{pro}
\label{pro:deterministic_QV}
Let $Y=(Y^n)_{n=1}^\infty$ be a sequence of centred Gaussian vectors such that $Y$ has finite energy. Then quadratic variation exists as a limit in probability and is deterministic if and only if the 2-planar variation is zero. In this case $\la Y\ra$ equals to the energy of the process and the convergence holds also in $L^p$ for any $p\geq 1$.
\end{pro}
\begin{proof}
The result follows directly from
$$
\lim_{n\to\infty}\E (Z_n - \varepsilon(Y))^2 = \E(Z_n^2) - \varepsilon(Y)^2
$$
and the relation \eqref{eq:key_relation} with $m=n$. Finally, the convergence in $L^p$ follows directly from Theorem \ref{thm:QV-ex}.
\end{proof}
\begin{rem}{\rm 
\label{rem:alpha-var}
A generalisation of quadratic variation is $\alpha$-variation which is defined as a limit of  
$
\sum_{k=1}^n \left|Y_k^{(n)}\right|^\alpha 
$ 
provided the limit exists. Similarly, we say that $Y$ has finite $\alpha$-energy if the limit
$$
\varepsilon_\alpha(Y) := \lim_{n\to\infty}\sum_{k=1}^n \E\left|Y_k^{(n)}\right|^\alpha
$$ 
exists. It is straightforward to show that if the sequence $Y=(Y_n)_{n=1}^\infty$ has finite $\alpha$-energy and 
$$
\sum_{k=1}^n \left|Y_k^{(n)}\right|^\alpha
$$
converges in probability, then the convergence holds also in $L^p$ for any $p\geq 1$. In this case however, the concept of 2-planar variation becomes much more complicated.
}\end{rem}
\subsection{Almost sure convergence}
In this subsection we address the question when the convergence takes place almost surely. 
The key ingredient for our results is the concentration inequality for Gaussian measures and by use of this inequality we show that, 
rather surprisingly, the quadratic variation always converges to the energy of the process whether or not the energy is finite provided that 2-planar variation vanishes.

Before stating our results we recall the following concentration inequality taken from \cite{b-h} for Gaussian processes. We present the result using our notation.
\begin{lma}
\label{lma:concentration}
Let $Y^n$ be a Gaussian vector with covariance matrix $\Gamma^{(n)}$, and denote $T = \sqrt{\text{trace}\left(\Gamma^{(n)}\right)}$. Then for any $h>0$ we have
\begin{equation}
\label{eq:concentration}
\P \left(\left|\Vert Y\Vert_2 - T\right| \geq h\right) \leq exp\left(-\frac{h^{2}}{4\Vert\Gamma^{(n)}\Vert_2}\right).
\end{equation}
\end{lma}
\begin{rem}{\rm 
The result holds for any Hilbert-valued Gaussian random variables, not only finite dimensional. 
Similarly, for general $\alpha$-variations one can use the concentration inequality
$$
\P\left(|\Vert X\Vert_{\mathcal{B}} - \E\Vert X\Vert_{\mathcal{B}}|\geq h\right) \leq \exp\left(-\frac{h^2}{2\sigma^2}\right),
$$
where $(\mathcal{B},\Vert\cdot\Vert_{\mathcal{B}})$ is a Banach space, $X$ is a Banach-valued Gaussian random variable and $\sigma^2 = \sup_{L\in \mathcal{B}':\Vert L\Vert \leq 1} \E L(X)^2$. Applying this to $\R^n$ equipped with the norm $\Vert\cdot\Vert_{\alpha}$ together with Riesz representation theorem one has
\begin{equation}
\label{eq:banach-concentration}
\P\left(|\Vert Y\Vert_{\alpha} - \E\Vert Y\Vert_{\alpha}|\geq h\right) \leq \exp\left(-\frac{h^2}{2\sigma^2}\right),
\end{equation}
where $\sigma^2 = \sup_{\Vert a\Vert_{q}\leq 1} \sum_{k=1}^n\sum_{j=1}^n a_ka_j\E[Y_kY_j]$ with $\frac{1}{\alpha}+\frac{1}{q}=1$.
}\end{rem}

We now turn to address the question when one obtains almost sure convergence. Clearly, the idea is to find an upper 
bound on $\Vert \Gamma^{(n)}\Vert_2$, say, $\Vert \Gamma^{(n)}\Vert_2 \leq\phi(n)$. Then one can hope that the bound is good enough such that
\begin{equation}
\label{as-convergence}
\sum_{n=1}^\infty exp\left(-\frac{\epsilon^{2}}{4\phi(n)}\right) < \infty
\end{equation}
from which we obtain the almost sure convergence immediately by Borel-Cantelli Lemma. 
\begin{thm}
\label{thm:QV-conv-general}
Let $Y=(Y^n)_{n=1}^\infty$ be a sequence of Gaussian vector such that $\Vert \Gamma^{(n)} \Vert_2 \rightarrow 0$ and
\begin{equation}
\label{eq:UI_general}
\sup_{n\geq 1}\sum_{k=1}^n \E \left(Y^{(n)}_k\right)^2 < \infty.
\end{equation}
Then, as $n\to \infty$,
\begin{equation}
\label{eq:gen_conv}
\left|\sum_{k=1}^n \left(Y^{(n)}_k\right)^2 - \sum_{k=1}^n \E \left(Y^{(n)}_k\right)^2\right| \rightarrow 0
\end{equation}
in probability and in $L^p$ for any $p\geq 1$. Furthermore, the convergence holds 
almost surely for any partition satisfying $\Vert \Gamma^{(n)} \Vert_2 =o\left(\frac{1}{\log n}\right)$. 
\end{thm}
\begin{proof}
The convergence From Lemma \ref{lma:concentration} we get that
$$
\left|\sqrt{\sum_{k=1}^n \left(Y^{(n)}_k\right)^2} - \sqrt{\sum_{k=1}^n \E \left(Y^{(n)}_k\right)^2}\right| \rightarrow 0
$$
in probability follows immediately from Lemma \ref{lma:concentration}, and the almost sure convergence follows by applying Borel-Cantelli Lemma. Now the convergence \eqref{eq:gen_conv} follows from decomposition $|a-b|=|\sqrt{a}-\sqrt{b}|(\sqrt{a}+\sqrt{b})$ together with the uniform integrability condition \eqref{eq:UI_general} which also implies convergence in $L^p$.
\end{proof}
\begin{rem}{\rm 
Note that if $\Vert \Gamma^{(n)}\Vert_2 \to 0$, we have
$$
\left|\sqrt{\sum_{k=1}^n \left(Y^{(n)}_k\right)^2} - \sqrt{\sum_{k=1}^n \E \left(Y^{(n)}_k\right)^2}\right| \rightarrow 0
$$
in probability even if $\sup_{n\geq 1}\sum_{k=1}^n \E \left(Y^{(n)}_k\right)^2 =\infty$, i.e. we have convergence in probability (or almost surely) while the energy might be infinite. For example, it can be shown that this is the case for fractional Brownian motion $B^H$ with Hurst index $H\in\left(0,\frac12\right)$ and its non-scaled quadratic variation, i.e. the one corresponding to a vector $Y^{(n)}_k = B^H_{t_k} - B^H_{t_{k-1}}$. 
}\end{rem}
\begin{rem}{\rm 
Note that for finite energy processes such that 2-planar variation vanishes this result shows that 
$$
\lim_{n,m\to\infty} \sum_{k=1}^n \sum_{j=1}^m\left[\E\left[Y_k^{(n)}Y_j^{(m)}\right]\right]^2 = \lim_{n\to\infty} \sum_{k=1}^n \sum_{j=1}^n \left[\E\left[Y_k^{(n)}Y_j^{(n)}\right]\right]^2
$$
or more compactly, $\Upsilon(Y) = \lim_{n\to\infty} \Vert \Gamma^{(n)}\Vert_F$. Indeed, from the well-known relation $\Vert A\Vert_2 \leq \Vert A \Vert_F$ we obtain that if $\lim_{n\to\infty}\Vert \Gamma^{(n)}\Vert_F \rightarrow 0$, then $\Vert \Gamma^{(n)}\Vert_2 \rightarrow 0$. 
Consequently, Theorem \ref{thm:QV-ex} implies that 2-planar variation vanishes. This answers to question raised in \cite[Remark 3.12]{ru-va2} in our discrete setting. Similarly, 
one can use general concentration inequality \eqref{eq:banach-concentration} to give analogous answer in a continuous case.
}\end{rem}
To compute the spectral norm $\Vert \Gamma^{(n)}\Vert_2$, or equivalently the largest eigenvalue, can be a challenging task. One way to overcome this challenge is to use Frobenius norm $\Vert\cdot \Vert_F$ which provides an upper bound and is easier to analyse. Unfortunately however, it provides quite rough estimates even in a simple case of standard Brownian motion as will be shown in 
subsection \ref{subsec:bm}. A way to obtain general conditions is to use matrix norm $\Vert \cdot \Vert_1$ which is also the main approach applied in the literature. This is the topic of the next general theorem. The proof is based on some well-known relations for matrix norm and we do not claim originality here.
\begin{thm}
Let $Y=(Y^n)_{n=1}^\infty$ be a sequence of Gaussian vectors such that \eqref{eq:UI_general} holds. Furthermore, assume there exists a function $\phi(n)$ such that 
$$
\max_j\sum_{k=1}^n \left|\E\left[Y_k^{(n)}Y_j^{(n)}\right]\right| \leq \phi(n).
$$
If $\phi(n) \to 0$, then the convergence
$$
\left|\sum_{k=1}^n \left[Y_k^{(n)}\right]^2 - \sum_{k=1}^n \E\left[Y_k^{(n)}\right]^2\right| \rightarrow 0
$$
holds in probability and in $L^p$ for any $p\geq 1$. Furthermore, if 
$$
\phi(n) = o\left(\frac{1}{\log n}\right),
$$
Then the convergence holds almost surely.
\end{thm}
\begin{proof}
Recall the well-known bound for matrix norm $\Vert A\Vert_2$ given by
$
\Vert A\Vert_2 \leq \sqrt{\Vert A\Vert_1 \Vert A\Vert_{\infty}},
$
where
$
\Vert A\Vert_1 = \max_{1\leq j\leq n} \sum_{k=1}^n |a_{kj}|
$
and
$
\Vert A\Vert_{\infty} = \max_{1\leq j\leq n} \sum_{k=1}^n |a_{jk}|.
$
Hence symmetry of $\Gamma^{(n)}$ gives
$
\Vert \Gamma^{(n)} \Vert_2 \leq \Vert \Gamma^{(n)} \Vert_1$, where 
$$
\Vert \Gamma^{(n)} \Vert_1 = \max_j\sum_{k=1}^n \left|\E\left[Y_k^{(n)}Y_j^{(n)}\right]\right|.
$$
This proves the claim together with Theorem \ref{thm:QV-conv-general}.
\end{proof}
The following final result of this section can be useful for stationary sequences.
\begin{thm}
Let $Y=(Y^n)_{n=1}^\infty$ be a sequence of Gaussian vectors such that \eqref{eq:UI_general} holds. Moreover, assume there exists a positive symmetric function $r$ such that 
$$
\left|\E\left[Y_k^{(n)}Y_j^{(n)}\right]\right| \leq r(k-j),\quad k,j=1,\ldots,n
$$
and assume that there exists a function $\phi(n)$ such that
$$
\sum_{k=1}^n r(k) \leq \phi(n).
$$
If $\phi(n)\to 0$ as $n$ tends to infinity, then the convergence 
$$
\left|\sum_{k=1}^n \left[Y_k^{(n)}\right]^2 - \sum_{k=1}^n \E\left[Y_k^{(n)}\right]^2\right| \rightarrow 0
$$
holds in probability and in $L^p$. Furthermore, the convergence holds almost surely provided that 
$
\phi(n) = o\left(\frac{1}{\log n}\right).
$
\end{thm}
\begin{proof}
Now for any $j\geq 1$ we have
\begin{equation*}
\sum_{k=1}^n \left|\E\left[Y_k^{(n)}Y_j^{(n)}\right]\right|
\leq \sum_{k=1}^n r(k-j) 
\leq \sum_{k=0}^{j-1}r(k) + \sum_{k=1}^{n-j}r(k)
\leq 2\sum_{k=0}^n r(k)
\end{equation*}
from which the result follows.
\end{proof}

\subsection{Central limit theorem}
In this section we provide sufficient and necessary condition for the central limit theorem (CLT) to hold. 
More importantly, we obtain two simple corollaries where the first one gives a version of Lindeberg's CLT for quadratic variations, and the second one gives simple to check
condition which actually holds in most of the studies cited in the introduction. In particular, the second corollary can be used to simplify the used techniques (namely, the Lindeberg's CLT)
considerably. Our necessary and sufficient condition is based on the fourth moment theorem by Nualart and Peccati \cite{nu-pe}. Hence we begin by recalling some basic facts on Wiener
chaos. For more details we refer to monographs \cite{nualart, pe-t, n-p}.

Let now $X$ be a separable Gaussian centered process and let $q\geq 1$ be fixed. The symbol $\mathcal{H}_{q}$ denotes the $q$th Wiener chaos of $X$, 
defined as the closed linear subspace of $L^2(\Omega)$
generated by the family $\{H_{q}(X(h)) : h\in  \HH,\left\| h\right\| _{ \HH}=1\}$, where $H_{q}$ is the $q$th Hermite polynomial. The mapping $I^{X}_{q}(h^{\otimes q})=H_{q}(X(h))$ can be extended to a
linear isometry between the symmetric tensor product $ \HH^{\odot q}$
and the $q$th Wiener chaos $\mathcal{H}_{q}$, and for any $h \in  \HH^{\odot q}$ the random variable $I^X_q(h)$ is called a multiple Wiener integral of order $q$. 
\begin{rem}{\rm 
If $X=W$ is a standard Brownian motion, then $\HH$ is simply the space $L^2([0,T],\d t)$. In this case the random variable 
$I^X_q(h)$ coincides with the $q$-fold multiple Wiener-It\^o integral of $h$ (see \cite{nualart}). 
}\end{rem}
\begin{rem}{\rm 
Let $X$ be a separable Gaussian process on $[0,T]$. It was proved in \cite{sot-vii} that $X$ admits a Fredholm integral representation
$$
X_t = \int_0^T K(t,s)\d W_s,
$$
where $K\in L^2([0,T]^2)$ and $W$ is a Brownian motion, if and only if $\int_0^T \E X_u^2 \d u < \infty$. Furthermore, it was shown that this representation can be extended to a 
transfer principle which can be used to develop stochastic calculus with respect to $X$. In particular, the transfer principle can be used to define multiple Wiener integrals as 
multiple Wiener integrals with respect to a standard Brownian motion. This definition coincides with the one defined via Hermite polynomials.
}\end{rem}
Finally, we are ready to recall the following characterisation of convergence towards a Gaussian limit.
\begin{thm}\cite{nu-pe}\label{thm:4th} 
Let $\{F_{n}\}_{n\geq 1}$ be a sequence of random variables in
the $q$th Wiener chaos, $q\geq 2$, such that $\lim_{n\rightarrow \infty } \E(F_{n}^{2})=\sigma ^{2}$. Then, as $n \to \infty$, the following asymptotic statements are equivalent:
\begin{itemize}
\item[(i)] $F_n$ converges in law to $\mathcal{N}(0,\sigma^2)$.
\item[(ii)] $\E F_n^4$ converges to $3 \sigma ^{4}$.
\end{itemize}
\end{thm}
\begin{rem}{\rm 
In this paper we are studying quadratic variations of Gaussian sequences. Hence, thanks to Fredholm representation from \cite{sot-vii}, such objects can be viewed as a sequences
in the second chaos. 
}\end{rem}
\begin{rem}{\rm 
The case of the second chaos was studied in details in \cite{no-po,no-po1} where the authors characterised all possible limiting laws. More precisely, authors in \cite{no-po} proved 
that if a sequence in the second chaos converges in law to some random variable $F$, then $F$ can be viewed as a sum of normal random variable and an independent random variable 
living in the second chaos.
}\end{rem}
\begin{rem}{\rm 
It was proved in \cite{n-o} that instead of fourth moment one can also study the convergence of $ \|DF_n\|^2_{\HH} $ in $L^2$, where $D$ stands for Malliavin derivative operator, 
and this approach have turned out to be very useful in some cases. For our purposes however it seems that working with the fourth moment is more convenient.
}\end{rem}
Finally, we recall the following Berry-Esseen type estimate taken from  \cite[Theorem 11.4.3]{pe-t}. 
\begin{thm}\label{thm:B-E-bound}
Let $\{F_n\}_{n \ge 1}$ be a sequence of elements in the $q$th Wiener chaos such that $\E(F_n^2) =1$ and let $Z$ denote a standard normal random variable. Then there exists a constant $C_q$ 
depending only on $q$ such that 
\begin{equation*}
 \sup_{x\in\R}\Big\vert \P(F_n < x) - \P(Z<x)\Big \vert \leq C_q\sqrt{\E F_n^4 - 3}.
\end{equation*}
\end{thm}

We are now ready to prove our results. We begin with the following auxiliary technical lemma whose proof is postponed to the appendix.
\begin{lma}
\label{lma:2nd_4th}
For $V_n$ given by \eqref{eq:V_n} we have
$$
\E V_n^2 = 2\sum_{k,j=1}^n \left(\E\left[Y^{(n)}_kY^{(n)}_j\right]\right)^2
$$
and
\begin{equation*}
\begin{split}
\E V_n^4 &= 12\left[\sum_{k,j=1}^n \left(\E\left[Y^{(n)}_kY^{(n)}_j\right]\right)^2\right]^2\\
&+ 24\sum_{i,j,k,l=1}^n \E\left[Y^{(n)}_kY^{(n)}_j\right]\E\left[Y^{(n)}_jY^{(n)}_i\right]\E\left[Y^{(n)}_iY^{(n)}_l\right]\E\left[Y^{(n)}_lY^{(n)}_k\right].
\end{split}
\end{equation*}
\end{lma}
With the help of this lemma we are ready to proof our main theorem.
\begin{thm}
\label{thm:CLT_iff}
Let $Y=(Y^n)_{n=1}^\infty$ be a sequence of Gaussian vectors such that for every $n\geq 1$ the elements $Y_k^{(n)},\quad k=1,\ldots, n$ belong to the first Wiener chaos, and 
let $\Gamma^{(n)}$ denote the covariance matrix of $Y^n$ with eigenvalues $\lambda_1^n, \lambda_2^n, \ldots, \lambda_n^n$. Then, as $n$ tends to infinity, the following are equivalent.
\begin{itemize}
\item $\frac{V_n}{\sqrt{Var(V_n)}} \rightarrow \mathcal{N}(0,1)$,
\item 
$$
\sum_{i,j,k,l=1}^n \E\left[Y^{(n)}_kY^{(n)}_j\right]\E\left[Y^{(n)}_jY^{(n)}_i\right]\E\left[Y^{(n)}_iY^{(n)}_l\right]\E\left[Y^{(n)}_lY^{(n)}_k\right] = o\left(Var(V_n)^2\right),
$$
\item
$$
\sum_{k=1}^n \left(\lambda_k^n\right)^4 =o\left(\left[\sum_{k=1}^n (\lambda_k^n)^2\right]^2\right).
$$
\end{itemize}
\end{thm}
\begin{proof}
By assumption we are able to use fourth moment theorem \ref{thm:4th} from which the equivalence of items (1) and (2) follows with the help of Lemma \ref{lma:2nd_4th}. 
To obtain equivalence of (1) and (3), it is well-known that $\sum_{k=1}^n \left[Y^{(n)}_k\right]^2$ can be decomposed as
$$
\sum_{k=1}^n \left[Y^{(n)}_k\right]^2 \law \sum_{k=1}^n \lambda_{k}^n \left[\xi_k^{(n)}\right]^2, 
$$
where $\lambda_k^n$ are the eigenvalues of the covariance matrix $\Gamma^{(n)}$ and $\xi^{(n)}_k$ are independent standard normal random variables. Denoting
$$
\tilde{V}_n = \sum_{k=1}^n \left[\lambda_{k}^n \left[\xi_k^{(n)}\right]^2- \lambda_{k}^n\right]
$$
and using Lemma \ref{lma:2nd_4th} again we obtain 
$$
\E \tilde{V}_n^2 = 2\sum_{k=1}^n [\lambda_k^n]^2
$$
and
$$
\E \tilde{V}_n^4 = 12\left[\sum_{k=1}^n [\lambda_k^n]^2\right]^2 + 6\sum_{k=1}^n [\lambda_k^n]^4
$$
which concludes the proof.
\end{proof}
As a simple corollary we obtain the following result which corresponds to Lindeberg's CLT and is mainly used in the references given in the introduction.
\begin{cor}
\label{cor:lindeberg}
Assume that assumptions of Theorem \ref{thm:CLT_iff} prevail. If
$$
\lambda^*(n):=\max_{k=1,\ldots,n}|\lambda_k^n| = o\left(\sqrt{Var(V_n)}\right), \quad n\to \infty,
$$
then
$$
\frac{V_n}{\sqrt{Var(V_n)}} \rightarrow \mathcal{N}(0,1).
$$
\end{cor}
\begin{proof}
We have 
$$
\sum_{k=1}^n (\lambda_k^n)^4 \leq \max_{k=1,\ldots,n}|\lambda_k^n|^2 \sum_{k=1}^n (\lambda_k^n)^2
$$
and since $Var(V_n) = \sum_{k=1}^n (\lambda_k^n)^2$, the result follows at once.
\end{proof}
\begin{rem}{\rm 
Note that since Lindeberg's CLT can be proved without the theory of Wiener chaos, the above result is valid for arbitrary sequences $Y^n$.
}\end{rem}
Finally, the following theorem justifies that in many cases it is sufficient to find an upper bound for $\lambda^*(n)$, or even 
for $\max_{1\leq j\leq n}\sum_{k=1}^n \left|\E\left[Y^{(n)}_kY^{(n)}_j\right]\right|$. While the proof follows from simple relations for matrix norm, the result turn out to be very useful for many practical applications. In particular, the following result covers many of the cases studied in the literature. Furthermore, in this case it easy to give a Berry-Esseen bound. 
\begin{thm}
\label{thm:CLT_main}
Let the assumptions of Theorem \ref{thm:CLT_iff} prevail, and assume that $Y^n$ is a Gaussian vector with finite non-zero energy. Then there exists a constant $C>0$ such that  
\begin{equation*}
 \sup_{x\in\R}\left| \P\left(\frac{V_n}{\sqrt{Var(V_n)}} < x\right) - \P(Z<x)\right| \leq C\sqrt{n}\lambda^*(n).
\end{equation*}
where $Z$ is a standard normal random variable. Hence if 
$$
\lambda^*(n) =o\left(\frac{1}{\sqrt{n}}\right),
$$
then $\frac{V_n}{\sqrt{Var(V_n)}} \rightarrow \mathcal{N}(0,1)$. In particular, 
if 
$$
\max_{1\leq j\leq n}\sum_{k=1}^n \left|\E\left[Y^{(n)}_kY^{(n)}_j\right]\right| = o\left(\frac{1}{\sqrt{n}}\right),
$$
then $\frac{V_n}{\sqrt{Var(V_n)}} \rightarrow \mathcal{N}(0,1)$. 
\end{thm}
\begin{proof}
Recall the trace norm is given by $\Vert \Gamma^{(n)} \Vert_* = \sum_{k=1}^n \lambda_k^n$. By Cauchy-Schwartz inequality we get the well-known matrix norm inequality 
$$
\Vert \Gamma^{(n)} \Vert_* \leq \sqrt{n}\sqrt{\sum_{k=1}^n [\lambda_k^n]^2} = \sqrt{n}\Vert \Gamma^{(n)} \Vert_F.
$$
By assumption, $Y^n$ has finite non-zero energy. Hence by observing that $\lim_{n\to\infty} \Vert \Gamma^{(n)}\Vert_* = \epsilon(Y) > 0$, we obtain that for large enough $n$ we have
$$
0<c\leq \sqrt{n}\Vert \Gamma^{(n)}\Vert_F.
$$
Next we observe that $\sqrt{Var(V_n)} = \Vert \Gamma^{(n)}\Vert_F$. Now the fourth moment of $\frac{V_n}{\sqrt{Var(V_n)}}$ is given by
$$
\frac{\E V_n^4}{\left(\E V_n^2\right)^2} = 3+\frac{3\sum_{k=1}^n [\lambda_k^n]^4}{2\left(\sum_{k=1}^n [\lambda_k^n]^2\right)^2}
$$
so that
$$
\frac{\E V_n^4}{\left(\E V_n^2\right)^2} - 3 \leq \frac{3[\lambda^*(n)]^2}{2\E V_n^2} \leq \frac{3}{2}n[\lambda^*(n)]^2.
$$
Hence the Berry-Esseen bound follows from Theorem \ref{thm:B-E-bound}.
\end{proof}
\begin{rem}{\rm
Note that the convergence towards normal random variable follow also from Corollary \ref{cor:lindeberg} which does not rely on the theory of Wiener chaos. However, for a sequence living in the second chaos we also obtain a Berry-Esseen bound.
}\end{rem}
\section{Application to quadratic variations}
\label{sec:qv}
In this section we apply the results to quadratic variations of Gaussian processes. We begin by giving a general results for generalised variations which will be illustrated in the 
particular cases of first and second order quadratic variations.

We consider arbitrary sequences of partitions $\pi_n = \{0=t^n_0 < t^n_1<\ldots < t^n_{N(\pi_n)-1}=T\}$, where $N(\pi_n)$ denotes the number of points in the partition. For notational simplicity, we drop the super-index $n$ and simply use $t_k$ instead of $t_k^n$.
The mesh of the partition is denoted by $|\pi_n|=\max\{t_k-t_{k-1}, t_k \in \pi_n \slash \{0\}\}$. 
We also use $m(\pi_n) = \min\{t_k-t_{k-1}, t_k \in \pi_n \slash \{0\}\}$. Throughout this section we assume that
\begin{equation}
\label{eq:partition_assumption}
\frac{|\pi_n|}{m(\pi_n)} \leq k(|\pi_n|), \quad n\geq 1
\end{equation}
for some function $k$. Obviously, usually the partition is chosen such that $k(|\pi_n|)\leq C < \infty$. 

\subsection{First order quadratic variations}
In this subsection we apply the results of previous section to study first order quadratic variations of Gaussian processes which is our main interest due to its connection to 
stochastic analysis. Throughout this subsection we also use the metric defined by the incremental variance of $X$, i.e.
$$
d_X(t,s) = \E (X_t - X_s)^2.
$$
\begin{dfn}
Let $X=(X_t)_{t\in[0,T]}$ be a centred Gaussian process. We say that $X$ has first order $\phi$-quadratic variation along $\pi_n$ if
\begin{equation}
V_1(\pi_n,\phi):=\sum_{t_k^n\pi_n} \frac{(X_{t_k}-X_{t_{k-1}})^2}{\phi(t_k-t_{k-1})}
\end{equation}
converges in probability as $|\pi_n|$ tends to zero. 
\end{dfn}
\begin{rem}{\rm 
A natural choice for the function $\phi$ is such that
\begin{equation}
\label{eq:fin_en_assumption}
\lim_{|\pi_n|\to 0}\sum_{k=1}^{N(\pi_n)-1} \frac{\E(X_{t_k}-X_{t_{k-1}})^2}{\phi(t_k-t_{k-1})} = K < \infty.
\end{equation}
In particular, in many interesting cases one has $d_X(t,s) \sim r(t-s)$ as $|t-s| \to 0$ for some function $r$. In this case a natural choice is $\phi(x)=\frac{r(x)}{x}$. 
}\end{rem}
To simplify the notation we denote $\tilde{V}_1(\pi_n,\phi) = V_1(\pi_n,\phi) - \E V_1(\pi_n,\phi)$. We also use $\Delta t_j = t_j - t_{j-1}$. 
We will begin by giving the following general theorem which generalises main results of \cite{malukas} by allowing us to drop some technical assumptions. 
The result follows directly by uniting and rewriting Theorems \ref{thm:QV-conv-general} and \ref{thm:CLT_main} for 
sequence $Y^{(n)}_k =   \frac{X_{t_k}-X_{t_{k-1}}}{\sqrt{\phi(t_k-t_{k-1})}}$. 
\begin{thm}
\label{thm:1st_order_main}
Let $X$ be a Gaussian process. Assume that 
\begin{equation}
\label{eq:general_condition}
\max_{1\leq j\leq N(\pi_n)-1}\sum_{k=1}^{N(\pi_n)-1} \frac{1}{\sqrt{\phi(\Delta t_k)\phi(\Delta t_j)}}|\E[(X_{t_k}-X_{t_{k-1}})(X_{t_j}-X_{t_{j-1}})]| \leq H(|\pi_n|)
\end{equation}
for some function $H(|\pi_n|)$. 
\begin{enumerate}
\item
If $H(|\pi_n|) \to 0$ as $|\pi_n|$ tends to zero, then convergence 
\begin{equation}
\label{eq:convergence}
\left|\sum_{k=1}^{N(\pi_n)-1} \frac{(X_{t_k}-X_{t_{k-1}})^2}{\phi(t_k-t_{k-1})} - \sum_{k=1}^{N(\pi_n)-1} \frac{\E(X_{t_k}-X_{t_{k-1}})^2}{\phi(t_k-t_{k-1})}\right| \rightarrow 0
\end{equation}
holds in probability. If $H(|\pi_n|) =o\left(\frac{1}{\log n}\right)$, then convergence  \eqref{eq:convergence} holds almost surely. In these cases the convergence holds also in $L^p$ for any $p\geq 1$ provided that \eqref{eq:fin_en_assumption} holds.
\item
Furthermore, assume that 
\begin{equation}
\label{eq:nontriv_energy}
\lim_{|\pi_n|\to 0}\sum_{k=1}^{N(\pi_n)-1} \frac{\E(X_{t_k}-X_{t_{k-1}})^2}{\phi(t_k-t_{k-1})} = K > 0.
\end{equation}
Then there exists a constant $C>0$ such that 
\begin{equation*}
 \sup_{x\in\R}\left| \P\left(\frac{\tilde{V}_1(\pi_n)}{\sqrt{Var(\tilde{V}_1(\pi_n))}} < x\right) - \P(Z<x)\right| \leq C\sqrt{N(\pi_n)}H(|\pi_n|),
\end{equation*}
where $Z$ is a standard normal random variable. 
In particular, if $H(|\pi_n|) = o\left(N(\pi_n)^{-\frac12}\right)$, then
$$
\frac{\tilde{V}_1(\pi_n)}{\sqrt{Var(\tilde{V}_1(\pi_n))}} \to \mathcal{N}(0,1).
$$
\end{enumerate}
\end{thm}
\begin{rem}{\rm 
In \cite{malukas} the author studied a particular class of Gaussian processes while here we consider arbitrary Gaussian process. Similarly, in \cite{malukas} the main 
result was derived by using some technical computations under assumption \eqref{eq:general_condition} together with several additional technical assumptions. 
Here we have shown that \eqref{eq:general_condition} is the only needed condition which generalises the class of processes considerably. Similarly, we have been able to simplify the 
proof since we have shown that such results follows essentially from Gaussian concentration together with some matrix algebra.
}\end{rem}
Next we will provide some sufficient conditions which are easy to check. A particularly interesting case for us is Gaussian processes such that the function
$$
d(s,t)= \E(X_t-X_s)^2
$$
is $C^{1,1}$ outside diagonal. 
Note that a sufficient condition for this is that the variance $\E X_t^2$ is $C^1$ and the covariance $R$ of $X$ is $C^{1,1}$ outside diagonal. 
Furthermore, note that this assumption is satisfied for many cases of interest. 

The first theorem gives a general result in a case of bounded derivative.
\begin{thm}
\label{thm:1st_order_bounded}
Let $X$ be a continuous Gaussian such that the function $d(s,t)=\E(X_t-X_s)^2$ is $C^{1,1}$ outside diagonal. Furthermore, assume that there exists a positive function $f(s,t)$ such that
$$
|\partial_{st}d(s,t)| \leq f(s,t)
$$
and
$$
\sup_{s,v\in[0,T]}\int_0^v f(s,t)\d t < \infty.
$$
Assume there exists a function $H(|\pi_n|)$ such that
$$
\max_{1\leq j\leq N(\pi_n)-1}\frac{d(t_j,t_{j-1})+\Delta t_j\sup_{1\leq k,j\leq N(\pi_n)-1}\frac{\sqrt{\phi(\Delta t_j)}}{\sqrt{\phi(\Delta t_k)}}}{\phi(\Delta t_j)} \leq H(|\pi_n|).
$$
Then the result of Theorem \ref{thm:1st_order_main} holds with function $H(|\pi_n|)$.
\end{thm}
\begin{proof}
Now for $j\neq k$ we have
$$
\E[(X_{t_k}-X_{t_{k-1}})(X_{t_j}-X_{t_{j-1}})] = \int_{t_{k-1}}^{t_k}\int_{t_{j-1}}^{t_j}\partial_{st}d(s,t)\d s\d t.
$$
Hence we have
\begin{equation*}
\begin{split}
&\sum_{k=1}^{N(\pi_n)-1} \left|\E\left[Y^{(n)}_k Y^{(n)}_j\right]\right|\\
&\leq\frac{d(t_j,t_{j-1})}{\phi(\Delta t_j)} + \frac{1}{\sqrt{\phi(\Delta t_j)}}\sum_{k=1, k\neq j}^{N(\pi_n)-1}\frac{1}{\sqrt{\phi(\Delta t_k)}}\int_{t_{k-1}}^{t_k}\int_{t_{j-1}}^{t_j}\left|\partial_{st}d(s,t)\right|\d s\d t\\
&\leq \frac{d(t_j,t_{j-1})}{\phi(\Delta t_j)} +\frac{1}{\phi(\Delta t_j)}\sum_{k=1, k\neq j}^{N(\pi_n)-1}\frac{\sqrt{\phi(\Delta t_j)}}{\sqrt{\phi(\Delta t_k)}}\int_{t_{k-1}}^{t_k}\int_{t_{j-1}}^{t_j}\left|\partial_{st}d(s,t)\right|\d s\d t\\
&\leq \frac{d(t_j,t_{j-1})}{\phi(\Delta t_j)} +\frac{\sup_{1\leq k,j\leq N(\pi_n)-1}\frac{\sqrt{\phi(\Delta t_j)}}{\sqrt{\phi(\Delta t_k)}}}{\phi(\Delta t_j)}\sum_{k=1, k\neq j}^{N(\pi_n)-1}\int_{t_{k-1}}^{t_k}\int_{t_{j-1}}^{t_j}\left|\partial_{st}d(s,t)\right|\d s\d t.
\end{split}
\end{equation*}
Now here
\begin{equation*}
\begin{split}
&\sum_{k=1,k\neq j}^{N(\pi_n)-1}\int_{t_{k-1}}^{t_k}\int_{t_{j-1}}^{t_j}\left|\partial_{st}d(s,t)\right|\d s\d t\\
&= \int_{0}^{t_{j-1}}\int_{t_{j-1}}^{t_j}\left|\partial_{st}d(s,t)\right|\d s\d t + \int_{t_{j}}^{T}\int_{t_{j-1}}^{t_j}\left|\partial_{st}d(s,t)\right|\d s\d t\\
&= O\left( \Delta t_j\right)
\end{split}
\end{equation*}
by Tonelli's theorem and assumptions. The claim follows at once. 
\end{proof}
\begin{rem}{\rm 
The convergence depends now on $\sup_{1\leq k,j\leq N(\pi_n)-1}\frac{\sqrt{\phi(\Delta t_j)}}{\sqrt{\phi(\Delta t_k)}}$. However, in any natural chosen sequence partition we have $\sup_{n\geq 1} k(\pi_n)<\infty$ which guarantees $\sup_{1\leq k,j\leq N(\pi_n)-1}\frac{\sqrt{\phi(\Delta t_j)}}{\sqrt{\phi(\Delta t_k)}} < \infty$ for many functions $\phi$. For example, this is obviously true for power functions $\phi(x) = x^\gamma$ which is a natural choice in many cases.
}\end{rem}
As an immediate corollary we obtain the following which again seems to generalise the existing results in the literature. Most importantly, 
the following result shows how the lower bound for the variance plays a fundamental role. 
Furthermore, a standard assumption in the literature is that $d(t,s) \sim |t-s|^\gamma$ for some number $\gamma\in(0,2)$ which in particular covers the fractional Brownian motion and related processes. In the following the structure of the variance can be more complicated. For simplicity we will only present the result in the case of bounded function $k(\pi_n)$ while the general case follows similarly.
\begin{cor}
\label{cor:1st_order_bounded}
Let the notation and assumptions above prevail. Furthermore, assume that there exists a function $r$ such that $d(t,s) \sim r(t-s)$ as $|t-s| \to 0$. Let $\sup_{n\geq 1}\sup_{1\leq k,j\leq N(\pi_n)-1}\frac{\sqrt{r(\Delta t_j)}}{\sqrt{r(\Delta t_k)}} < \infty$, $\sup_{n\geq 1} k(\pi_n)<\infty$, and put $\phi(\Delta t_j) = \frac{r(\Delta t_j)}{\Delta t_j}$. 
\begin{enumerate}
\item
If $(|\pi_n|)^2 =o\left(r(|\pi_n|)\right)$, then
$$
|\tilde{V}_1(\pi_n,\phi)| \to 0
$$
in probability and in $L^p$ for any $p\geq 1$. The convergence holds almost surely for any sequence $\frac{(|\pi_n|)^2}{r(|\pi_n|)} = o\left(\frac{1}{\log n}\right)$.
\item
There is a constant $C>0$ such that 
\begin{equation*}
 \sup_{x\in\R}\left| \P\left(\frac{\tilde{V}_1(\pi_n)}{\sqrt{Var(\tilde{V}_1(\pi_n))}} < x\right) - \P(Z<x)\right| \leq C\frac{|\pi_n|^{\frac32}}{r(|\pi_n|)},
\end{equation*}
where $Z$ is a standard normal random variable.
In particular, if $(|\pi_n|)^{\frac32} =o\left(r(|\pi_n|)\right)$, then
$$
\frac{\tilde{V}_1(\pi_n)}{\sqrt{Var(\tilde{V}_1(\pi_n))}} \to \mathcal{N}(0,1).
$$
\end{enumerate}
\end{cor}
\begin{proof}
The claim follows immediately from Theorem \ref{thm:1st_order_bounded} by noting that our choice of function $\phi$ guarantees condition \eqref{eq:nontriv_energy}.
\end{proof}
We end this section with a following result that recovers the case of fractional Brownian motion and related processes. 
Note again that our technical assumptions are quite minimal.
\begin{thm}
\label{thm:1st_order_fbm}
Let $X$ be a continuous Gaussian such that the function $d(s,t)=\E(X_t-X_s)^2$ is in $C^{1,1}$ outside diagonal. Furthermore, assume that 
$$
|\partial_{st}d(s,t)| =O\left(|t-s|^{2H-2}\right)
$$
for some $H\in(0,1)$, $H\neq \frac12$ and assume there exists a function $H(|\pi_n|)$ such that
$$
\max_{1\leq j\leq N(\pi_n)-1}\frac{d(t_j,t_{j-1})+[\Delta t_j]^{\min(1,2H)}\sup_{1\leq k,j\leq N(\pi_n)-1}\frac{\sqrt{\phi(\Delta t_j)}}{\sqrt{\phi(\Delta t_k)}}}{\phi(\Delta t_j)} \leq H(|\pi_n|).
$$
Then the result of Theorem \ref{thm:1st_order_main} holds with function $H(|\pi_n|)$. 
\end{thm}
\begin{proof}
Note that the case $H>\frac12$ follows directly from Theorem \ref{thm:1st_order_bounded}. Let now $H<\frac{1}{2}$.
Using Fubini's Theorem we have 
\begin{equation*}
\begin{split}
&\int_{t_j}^{T}\int_{t_{j-1}}^{t_j}\left|\partial_{st}d(s,t)\right|\d s\d t\\
&\leq C\int_{t_{j-1}}^{t_j} (T-s)^{2H-1}\d s + C\int_{t_{j-1}}^{t_j} (t_j-s)^{2H-1}\d s\\
&\leq C [\Delta t_j]^{2H}
\end{split}
\end{equation*}
for some unimportant constants $C$ which vary from line to line. Here we have used the fact that for positive numbers $a,b$ and $\gamma \in (0,1)$ we have 
$|a^\gamma - b^\gamma| \leq |a-b|^\gamma$. Treating integral $\int_{0}^{t_{j-1}}\int_{t_{j-1}}^{t_j}\left|\partial_{st}d(s,t)\right|\d s\d t$ similarly the result follows.
\end{proof}
\begin{rem}{\rm 
\label{rem:case_half}
We remark that the case $H=\frac12$ can be treated similarly. In this case one obtains a condition
$$
\max_{1\leq j\leq N(\pi_n)-1}\frac{d(t_j,t_{j-1})+\Delta t_j |\log \Delta t_j|\sup_{1\leq k,j\leq N(\pi_n)-1}\frac{\sqrt{\phi(\Delta t_j)}}{\sqrt{\phi(\Delta t_k)}}}{\phi(\Delta t_j)} \leq H(|\pi_n|).
$$
}\end{rem}
\begin{rem}{\rm 
It is straightforward to give a version of Corollary \ref{cor:1st_order_bounded} in this case also. Indeed, the only difference is slightly different exponents in the case $H<\frac12$.
}\end{rem}
\subsection{Second order quadratic variations}
In this subsection we briefly study second order quadratic variations. In particular, we reproduce and generalise the results presented in papers \cite{begyn-ejp} and \cite{begyn-ber}. 
We will present our results in slightly different form. However, comparison is provided in remark \ref{rem:comparison}.

Usually second order quadratic variation on $[0,1]$ is defined as the limit of $\sum_{k=1}^n \left(X_{\frac{k+1}{n}} - 2X_{\frac{k}{n}} +X_{\frac{k-1}{n}}\right)$. 
To generalise for irregular
subdivisions Begyn introduced and motivated \cite{begyn-ejp} a second order differences along a sequence $\pi_n$ as 
$$
\Delta X_k = \Delta t_{k+1} X_{t_{k-1}} + \Delta t_{k} X_{t_{k+1}} - \left(\Delta t_{k+1} + \Delta t_k\right)X_{t_k},
$$
and we study the second order quadratic variation defined as the limit 
$$
V_2(\pi_n) = \sum_{k=1}^{N(\pi_n) -1} \frac{\Delta t_{k+1} \left(\Delta X_k\right)^2}{\E \left(\Delta X_k\right)^2}.
$$
Again we use short notation 
$$
\tilde{V}_2(\pi_n) = V_2(\pi_n) - \E V_2(\pi_n).
$$
We also assume that the derivative $\frac{\partial^4}{\partial^2 s\partial^2 t}R(s,t)$ of the covariance function $R$ of $X$ exists outside diagonal and satisfies 
\begin{equation}
\label{eq:2nd_order_partial}
\left|\frac{\partial^4}{\partial^2 s\partial^2 t}R(s,t)\right| \leq \frac{C}{|t-s|^{2+\gamma}}
\end{equation}
for some number $\gamma \in(0,2)$. Finally, we make the simplifying assumption $\sup_n k(\pi_n) < \infty$ on the function $k$. Hence it is also natural to assume
\begin{equation}
\label{eq:2nd_order_var_bound}
\sup_{j,k}\frac{\E \left(\Delta X_k\right)^2}{\E \left(\Delta X_k\right)^2} < \infty.
\end{equation}
In particular, the assumptions made in \cite{begyn-ejp} implied 
\begin{equation}
\label{eq:2nd_order_var}
\E \left(\Delta X_k\right)^2 \sim \left(\Delta t_{k+1}\right)^{\frac{3-\gamma}{2}}\left(\Delta t_{k}\right)^{\frac{3-\gamma}{2}}\left(\Delta t_{k+1}+\Delta t_{k}\right)
\end{equation}
in which case \eqref{eq:2nd_order_var_bound} is clearly satisfied.
\begin{thm}
Let all the notation and assumptions above prevail. 
\begin{enumerate}
 \item 
If $H(\pi_n):=\max_{1\leq j\leq N(\pi_n)-1}\frac{|\pi_n|^{5-\gamma}}{\E \left(\Delta X_k\right)^2}$ converges to zero, then
$$
\left|\tilde{V}_2(\pi_n)\right| \to 0
$$
in probability and in $L^p$ for any $p\geq 1$. Furthermore, the convergence holds almost surely provided that $H(\pi_n) =o\left(\frac{1}{\log n}\right)$.
\item
There is a constant $C>0$ such that 
\begin{equation*}
 \sup_{x\in\R}\left| \P\left(\frac{\tilde{V}_2(\pi_n)}{\sqrt{Var(\tilde{V}_2(\pi_n))}} < x\right) - \P(Z<x)\right| \leq C\sqrt{N(\pi_n)}H(|\pi_n|),
\end{equation*}
where $Z$ is a standard normal random variable.
In particular, if $H(\pi_n) =o\left(\left[N(\pi_n)\right]^{-\frac12}\right)$, then
$$
\frac{\tilde{V}_2(\pi_n)}{\sqrt{Var\left(\tilde{V}_2(\pi_n)\right)}} \to \mathcal{N}(0,1).
$$
\end{enumerate}
\end{thm}
\begin{proof}
Denote
$$
Y^{(n)}_k = \frac{\sqrt{\Delta t_k}\Delta X_k}{\sqrt{\E \left(\Delta X_k\right)^2}}.
$$
For some constant $C$ we have
$$
\sum_{k=1}^{N(\pi_n)-1} \left|\E \left[Y^{(n)}_kY^{(n)}_j\right]\right| \leq C \frac{|\pi_n|}{\E \left(\Delta X_j\right)^2} \sum_{k=1}^{N(\pi_n)-1} \left|\E \left[\Delta X_k \Delta X_j\right]\right|
$$
by \eqref{eq:2nd_order_var_bound} and boundedness of $k(\pi_n)$. Furthermore, it was proved in \cite{begyn-ejp} that boundedness of $k(\pi_n)$ together with 
\eqref{eq:2nd_order_partial} yields
$$
\sum_{k=1}^{N(\pi_n)-1} \left|\E \left[\Delta X_k \Delta X_j\right]\right| \leq C|\pi_n|^{4-\gamma}
$$
for a constant $C$.
This gives a bound for 
$$
\sum_{k=1}^{N(\pi_n)-1} \left|\E \left[Y^{(n)}_kY^{(n)}_j\right]\right|
$$
from which the result follows immediately by combining Theorems \ref{thm:QV-conv-general} and \ref{thm:CLT_main}.
\end{proof}
\begin{rem}{\rm 
\label{rem:comparison}
To compare our result with the ones provided in papers \cite{begyn-ejp} and \cite{begyn-ber}, first note that we were able to reproduce and generalise the main theorem of \cite{begyn-ejp} 
although we gave our result in a slightly different form. Indeed, in \cite{begyn-ejp} several additional technical conditions were assumed to ensure the asymptotic relation 
\eqref{eq:2nd_order_var} while here we have worked with general variance. This is helpful since the message of our result is that basically one has to only check the asymptotic 
behaviour of the variance, and \eqref{eq:2nd_order_partial} quarantees the upper bound for the corresponding matrix norm. Furthermore, the central limit theorem in \cite{begyn-ber}
was proved only for uniformly divided partition, and the central limit theorem was proved under more restrictive technical conditions, similar to those in \cite{begyn-ejp}, 
and by finding a lower bound for the variance in order to apply Lindeberg's CLT. Here we have proved that such result holds also non-uniform partitions and the result follows easily 
from the computations presented in \cite{begyn-ejp} together with Theorem \ref{thm:CLT_main}. Finally, here we also obtained Berry-Esseen bound. In particular, under \eqref{eq:2nd_order_var} we obtain bound proportional to $\sqrt{|\pi_n|}$.
}\end{rem}
\subsection{Remarks on generalised quadratic variations}
\label{subsec:a-qv}
We end this section by giving some remarks on generalised quadratic variations introduced by Istas and Lang \cite{I-L}.

Let now $a=(a_0,a_1,\ldots,a_p)$ be a vector such 
that $\sum_{k=0}^p a_k = 0$, where $p$ is a fixed integer. Let also $\Delta$ be a fixed small number and consider time points
$t_k = k\Delta,\quad k=1,\ldots, n$. The $a$-differences of $X$ is given by 
$$
\Delta_a X_j = \sum_{k=0}^p a_k X_{t_{j+k}}, \quad j=0,1,\ldots,n - p.
$$
In \cite{I-L} the authors considered stationary or stationary increment 
Gaussian processes and studied generalised $a$-variations defined as a limit of 
$$
V(a,n) = \frac{1}{n}\sum_{k=1}^{n-p} \frac{(\Delta_a X_k)^2}{\E [\Delta_a X_k]^2}.
$$
Now since $X$ is either stationary or has stationary increments, the function $d(t,s)$ depends only on the difference $|t-s|$. The assumption in \cite{I-L} was that 
the function 
$v(t) = d(0,t)$ is $2D$ times differentiable ($D$ is the greatest such integer, possibly 0), and for some number $\gamma\in(0,2)$ and some constant $C>0$ we have
$$
v^{(2D)}(t) = v^{(2D)}(0) + Ct^{\gamma} + r(t),
$$
where the remainder $r$ satisfies $r(t) = o\left(t^{\gamma}\right)$. The main results in \cite{I-L} was that under different set of assumptions one can obtain Gaussian limit with 
some rate with a suitable choice of the vector $a$, although in some cases one has to assume the observation window $n\Delta$ to increase to infinity. 
Obviously, by using the result of this paper we could reproduce and generalise, at least to cover more general variances as 
in here for first and second kind quadratic variations, these 
results together with a much simplified proofs. 
Instead of getting lost into technical details we wish to give some remarks and explanations. In \cite{I-L} the main message was roughly that larger the value of 
$D$ and $s$, then larger one has to choose the value $p$, i.e. taking account more refined discretisation. However, the reason for this can, and at least was for the present author, 
be lost in the technicalities. In a nutshell, the idea is to try to find a discretisation vector $a$ so that 
$$
\max_{1\leq j\leq n}\frac{1}{n^2\sqrt{\E [\Delta_a X_j]^2}}\sum_{k=1}^{n-p}\frac{1}{\sqrt{\E [\Delta_a X_k]^2}}\left|\E\left[(\Delta_a X_k)(\Delta_a X_j)\right]\right| = o\left(n^{-\frac12}\right)
$$
from which we obtain almost sure convergence and central limit theorem (with $\sqrt{Var\left(V(a,n)-\E V(a,n)\right)}$ as normalisation so one is left to analyse the asymptotic of 
this variance). Hence it remains to answer how one should choose the vector $a$. The idea for this comes little bit clearer from the first order variation and Corollary \ref{cor:1st_order_bounded}. 
Indeed, as also pointed out in \cite{I-L}, the number $D$ is the order of
differentiability of $X$ in the $L^2$-sense. Hence if $D\geq 1$, the variance must be at least of order $(\Delta t_j)^2$ so there is no hope to obtain even convergence
in probability. Hence larger the $D$, larger the value of $p$ should also be. Similarly, as $\gamma$ becomes close to number $2$ it roughly means that $D$ comes closer to $1$ so
once again one needs to refine the discretisation to obtain a Gaussian limit. More precisely, as $\gamma$ comes closer to $2$ we see immediately that the variance is no longer enough to compensate $|\pi_n|^{\frac32}$ in order to obtain central limit theorem. Hence one has to consider second order quadratic variations in order to stay in a Gaussian paradise.

\section{Examples}
\label{sec:ex}
This section is devoted to examples. We focus to reproduce some interesting and already studied examples to illustrate the power of our method rather than finding a complicated new examples, while now already it should be obvious how our approach can be used to study more complicated cases. More precisely, we study Brownian motion, fractional Brownian motion and related processes together with extensions sub-fractional Brownian motion and bifractional Brownian motion. Furthermore, our main focus is on first order quadratic variation for which we find sufficient conditions for the mesh to obtain almost sure convergence. In this context particularly interesting case for us is bifractional Brownian motion for which we are able to improve the sufficient condition proved in \cite{malukas}. For simplicity we assume that the function $k(\pi_n)$ is bounded.
\subsection{Standard Brownian motion}
\label{subsec:bm}
Let $X=W$ be a standard Brownian motion. Then it is known that the almost sure convergence holds provided that $|\pi_n| = o\left(\frac{1}{\log n}\right)$ (for recent results on the topic see \cite{le-er}). Furthermore, this is a sharp in a sense that one can construct a sequence with $|\pi_n| = O\left(\frac{1}{\log n}\right)$ such that almost sure convergence does not hold. Now the sufficiency of $|\pi_n| = o\left(\frac{1}{\log n}\right)$ follows easily 
from concentration inequality \eqref{eq:concentration} applied to the increments of Brownian motion. Indeed, in the case of standard Brownian motion the covariance matrix $\Gamma^{(n)}$ of the increments is diagonal, and we have
$$
\Vert \Gamma^{(n)}\Vert_1 = \Vert \Gamma^{(n)} \Vert_2 = |\pi_n|.
$$
Note also that if one uses Frobenius norm $\Vert\Gamma^{(n)}\Vert_F$ to obtain the upper bound, we have $\Vert \Gamma^{(n)}\Vert_F = \sqrt{|\pi_n|}$ provided that $\frac{|\pi_n|}{m(\pi_n)} \leq C$. Hence by using Frobenius norm we can only obtain half of the best possible rate even in the case of standard Brownian motion. Finally, it is straightforward to obtain central limit theorem which, of course, is well-known already.
\subsection{Fractional Brownian motion}
Recall that a fractional Brownian motion $B^H$ with Hurst index $H\in(0,1)$ is a continuous centred Gaussian process with covariance function
$$
R(t,s) = \frac{1}{2}\left(t^{2H}+s^{2H}-|t-s|^{2H}\right).
$$
The case $H=\frac12$ reduces to a standard Brownian motion. To obtain $L^p$-convergence of general $\alpha$-variations is straightforward by using Remark \ref{rem:alpha-var}.
\begin{pro}
Let $B^H$ be a fractional Brownian motion with $H\in(0,1)$. Then there exists a constant $C_H$ such that for $\alpha=\frac{1}{H}$ we have
$$
\sum_{t^n_k\in \pi_n} |B_{t_k}^H - B_{t_{k-1}}^H|^\alpha \rightarrow C_H T
$$
in $L^p$ for any $p\geq 1$. 
\end{pro}
\begin{rem}{\rm 
The exact value of the constant $C_H$ is given by $C_H = \E|N|^\frac{1}{H}$, where $N$ is a standard normal random variable.
}\end{rem}
We now turn to the convergence of quadratic variation which is more interesting for us.
Now it is natural to take $\phi(x) = x^{2H-1}$, since for any partition of $[0,T]$ we have
$$
\sum_{t^n_k\in \pi_n} \frac{\E(B_{t_k}^H - B_{t_{k-1}}^H)^2}{[\Delta t_k]^{2H-1}} = T.
$$
The following result is a direct consequence of Theorem \ref{thm:1st_order_fbm}. 
\begin{pro}
\label{pro:fbm}
Let $B^H$ be a fractional Brownian motion with $H\in(0,1)$. Then
\begin{enumerate}
\item
$$
V_n^B:=\sum_{k=1}^{N(\pi_n)-1} \frac{(B_{t_k}^H - B_{t_{k-1}}^H)^2}{[\Delta t_k]^{2H-1}} \to T
$$ 
in probability and in $L^p$ for any $p\geq 1$. Furthermore, the convergence holds almost surely for any sequence of partitions satisfying $|\pi_n| =o\left(\frac{1}{(\log n)^\gamma}\right)$, where $\gamma = \max\left(\frac{1}{2-2H},1\right)$. 
\item
There exists a constant $C>0$ such that
\begin{equation*}
 \sup_{x\in\R}\left| \P\left(\frac{V_n^B - T}{\sqrt{Var(V_n^B - T)}} < x\right) - \P(Z<x)\right| \leq C|\pi_n|^{\min\left(\frac12,\frac32-2H\right)},
\end{equation*}
where $Z$ is a standard normal random variable.
In particular, central limit theorem holds for all values $H<\frac{3}{4}$.
\end{enumerate}
\end{pro}
\begin{rem}{\rm 
Note that in the case $H<\frac12$ we obtain similar sufficient condition as for standard Brownian motion. Indeed, the only difference is that since here the increments are not independent, we have to pose an additional assumption $\sup_{n\geq 1} k(\pi_n)<\infty$ to obtain the optimal condition $o\left(\frac{1}{\log n}\right)$. 
}\end{rem}
\begin{rem}{\rm 
By considering uniform partitions it can be shown that via concentration inequalities one cannot obtain any better result. It would be interesting to know whether the given conditions are optimal similarly as in the case of Brownian motion. However, for Brownian motion the counter-examples are constructed relying on independence of increments and, to the best of our knowledge, there exists no method to attack the problem for general Gaussian process.
}\end{rem}
\begin{rem}{\rm 
The limit theorems for quadratic variations of fractional Brownian motion are extensively studied in the literature. However, most of the related studies rely on uniform partitions and focus on generalisations, e.g. to study Hermite variations or weighted variations rather than generalising the sequence of partitions. Furthermore, to recover the central limit theorem in the case $H<\frac34$ our approach is based only on simple linear algebra. For this reason our approach may be more applicable to generalise the results for arbitrary Gaussian processes while the obvious drawback is that it cannot provide a full answer to the problem. Finally, to the best of our knowledge the non-uniform partition are not widely studied in the literature. 
}\end{rem}
\begin{rem}{\rm 
It is known that in the critical case $H=\frac34$ we also have convergence towards a normal random variable (see, e.g. \cite{nou-rev} and references therein) with the only difference in the rate, i.e. the variance is of order $\frac{\log n}{n}$ instead of $\frac{1}{n}$. Given a priori knowledge that the variance is of order $\frac{\log n}{n}$ it is straightforward to recover also this critical case by Corollary \ref{cor:lindeberg}. Hence again, it is sufficient to study the asymptotic behaviour of the variance. 
}\end{rem}
\subsection{sub-fractional Brownian motion}
The sub-fractional Brownian motion $G^H$ with parameter $H\in(0,1)$ is a centred Gaussian process with covariance function
$$
R(s,t) = s^{2H} + t^{2H}-\frac{1}{2}\left[s^{2H}+t^{2H}-|t-s|^{2H}\right].
$$
Note that as for fractional Brownian motion, value $H=\frac12$ corresponds to a standard Brownian motion. 
\begin{pro}
Let $G^H$ be a sub-fractional Brownian motion with $H\in(0,1)$. Then
\begin{enumerate}
\item
$$
V_n^G:=\sum_{t^n_k\in \pi_n} \frac{\left(G_{t_k}^H - G_{t_{k-1}}^H\right)^2}{[\Delta t_k]^{2H-1}}\to T
$$
in probability and in $L^p$ for any $p\geq 1$. Furthermore, the convergence holds almost surely for any sequence satisfying $|\pi_n| =o\left(\frac{1}{(\log n)^\gamma}\right)$, where $\gamma = \max\left(\frac{1}{2-2H},1\right)$. 
\item
There exists a constant $C>0$ such that 
\begin{equation*}
 \sup_{x\in\R}\left| \P\left(\frac{V_n^G - T}{\sqrt{Var(V_n^G - T)}} < x\right) - \P(Z<x)\right| \leq C|\pi_n|^{\min\left(\frac12,\frac32-2H\right)},
\end{equation*}
where $Z$ is a standard normal random variable.
In particular, central limit theorem holds for all values $H<\frac{3}{4}$.
\end{enumerate}
\end{pro}
\begin{proof}
Note that the case $H=\frac12$ is already covered since it reduces back to the standard Brownian motion. Hence assume $H\neq \frac12$ and let $t>s$. We have 
$$
\partial_{st}d_G(t,s) \leq C (t+s)^{2H-2} +C|t-s|^{2H-2}
$$
for some constant $C$ and $d_G(s,t) \leq C|t-s|^{2H}$. Now 
$
(t+s)^{2H-2} \leq |t-s|^{2H-2}
$
from which the result follows immediately as in the case of fractional Brownian motion.
\end{proof}
\begin{rem}{\rm 
We remark that the above result was already given in \cite{malukas} with the same rates although there the condition for the case $H=\frac12$ was $|\pi_n||\log|\pi_n|| = o\left(\frac{1}{\log n}\right)$ which would follow from Remark \ref{rem:case_half}. Obviously however, in this case we have standard Brownian motion so that $|\pi_n| = o\left(\frac{1}{\log n}\right)$ is sufficient. Note also that in this case one cannot obtain better via concentration inequalities. Indeed, this comes from the ''fractional Brownian part'' $|t-s|^{2H}$. 
}\end{rem}
\subsection{Bifractional Brownian motion}
Particularly interesting case for us is Bifractional Brownian motion which was also studied in \cite{malukas}. 
However, with our method we are able to improve the results of \cite{malukas}.

The bifractional Brownian motion is an extension of fractional Brownian motion first introduced by \cite{h-v} and later analysed e.g. by \cite{ru-tu}. 
Finally, bifractional Brownian motion was extended for values $K\in(1,2)$ such that $HK\in(0,1)$ by \cite{bar-seb}.
\begin{dfn}
The bifractional Brownian motion is a centred Gaussian process $B^{H,K}$ with $B^{H,K}_0=0$ and covariance function
$$
R(t,s) = \frac{1}{2^K}\left((t^{2H}+s^{2H})^K-|t-s|^{2HK}\right)
$$
with $H\in(0,1)$ and $K\in(0,2)$ such that $HK\in(0,1)$.
\end{dfn}
\begin{rem}{\rm 
Note that $K=1$ corresponds to ordinary fractional Brownian motion. It is straightforward to check that $B^{H,K}$ is $HK$-self-similar and Hölder continuous of any order $\gamma<HK$. For more details on the properties of bifractional Brownian motion we refer to \cite{ru-tu} and references therein.
}\end{rem}
While the main emphasis in \cite{ru-tu} was integration via regularisation it was pointed out that one can prove that $\alpha$-variations exists as a limit in $L^1$ in our sense. Hence the following result is obvious from remark \ref{rem:alpha-var}.
\begin{pro}
Let $B^{H,K}$ be a bifractional Brownian motion with $H\in(0,1)$ and $K\in(0,2)$ such that $HK\in(0,1)$. Then there exists a constant $C_{H,K}$ such that for $\alpha=\frac{1}{HK}$ we have
$$
\sum_{t^n_k\in \pi_n} \left|B_{t_k}^{H,K} - B_{t_{k-1}}^{H,K}\right|^\alpha \rightarrow C_{H,K} T
$$
in $L^p$ for any $p\geq 1$. 
\end{pro}
\begin{rem}{\rm 
In \cite{ru-tu} the authors considered only the case $K\in(0,1]$. However, it is straightforward to obtain the claim for the case $K>1$ by repeating the arguments. 
}\end{rem}
The next theorem studies the quadratic variation of bifractional Brownian motion.
\begin{pro}
Let $B^{H,K}$ be a bifractional Brownian motion with $H\in(0,1), K\in(0,2)$ and $HK\in(0,1)$. Then
\begin{enumerate}
\item
$$
V_n^{H,K}:= \sum_{t^n_k\in \pi_n} \frac{\left(B_{t_k}^{H,K} - B_{t_{k-1}}^{H,K}\right)^2}{[\Delta t_k]^{2HK-1}} \to 2^{1-K}T
$$
in probability and in $L^p$ for any $p\geq 1$. Furthermore, the convergence holds almost surely for any sequence of partitions satisfying 
$|\pi_n| =o\left(\frac{1}{(\log n)^\gamma}\right)$, where;
\begin{itemize}
\item
$\gamma = \max\left(\frac{1}{2-2HK},1\right)$ for $K\in(0,1]$,
\item
$\gamma = \frac{1}{\min(1,2H)+1-2HK}$ for $K\in(1,2)$.
\end{itemize}
\item
In the case $K\in(0,1]$ we have
\begin{equation*}
 \sup_{x\in\R}\left| \P\left(\frac{V_n^{H,K} - 2^{1-K}T}{\sqrt{Var(V_n^{H,K} - 2^{1-K}T)}} < x\right) - \P(Z<x)\right| \leq C|\pi_n|^{\min\left(\frac12,\frac32-2HK\right)}
\end{equation*}
for some constant $C$ and a standard normal random variable $Z$.
In particular, central limit theorem holds for all values $HK<\frac{3}{4}$.
\item
In the case $K\in(1,2)$ we have
\begin{equation*}
 \sup_{x\in\R}\left|\P\left(\frac{V_n^{H,K} - 2^{1-K}T}{\sqrt{Var(V_n^{H,K} - 2^{1-K}T)}} < x\right) - \P(Z<x)\right|\leq C|\pi_n|^{\min\left(1,2H\right)-2HK+\frac12}.
\end{equation*}
\end{enumerate}
\end{pro}
\begin{proof}
We assume $K\neq 1$ since the case $K=1$ reduces to ordinary fractional Brownian motion treated in Proposition \ref{pro:fbm}.

The function $d(s,t) = \E[B_t - B_s]^2$ is differentiable outside diagonal and we have
$$
\partial_{st}d(s,t) = C_1 |t-s|^{2HK-2} + C_2 \frac{(ts)^{2H-1}}{(t^{2H}+s^{2H})^{2-K}}
$$
for some unimportant constants $C_1$ and $C_2$. Furthermore, we have
$$
d(s,t) \sim |t-s|^{2HK}
$$
as $|t-s|\to 0$ which also corresponds to fractional Brownian motion. Now the term $|t-s|^{2HK-2}$ can be treated as in Theorem \ref{thm:1st_order_fbm}, and for $HK=\frac12$ this term vanishes. Consider next the term
$$
\frac{(ts)^{2H-1}}{(t^{2H}+s^{2H})^{2-K}}.
$$
We are left to bound integrals
$$
I_1 = \int_0^{t_{j-1}}\int_{t_{j-1}}^{t_j} \frac{(ts)^{2H-1}}{(t^{2H}+s^{2H})^{2-K}}\d s \d t
$$
and
$$
I_2 = \int_{t_{j}}^{T}\int_{t_{j-1}}^{t_j} \frac{(ts)^{2H-1}}{(t^{2H}+s^{2H})^{2-K}}\d s \d t.
$$
We consider only $I_1$ since $I_2$ can be treated similarly, and we denote by $C$ any unimportant constant which may vary from line to line. By change of variables $x = t^{2H}$, $y=s^{2H}$ and Tonelli's theorem we have
\begin{equation*}
\begin{split}
I_1 &= C\int_0^{t_{j-1}^{2H}}\int_{t_{j-1}^{2H}}^{t_j^{2H}} \frac{(xy)^{\frac{1}{2H}(2H-1)}}{(x+y)^{2-K}}x^{\frac{1}{2H}-1}y^{\frac{1}{2H}-1}\d x \d y\\
&=C\int_0^{t_{j-1}^{2H}}\int_{t_{j-1}^{2H}}^{t_j^{2H}}(x+y)^{K-2}\d x \d y\\
&=C\int_{t_{j-1}^{2H}}^{t_j^{2H}} \int_0^{t_{j-1}^{2H}}(x+y)^{K-2}\d x \d y\\
&=C\int_{t_{j-1}^{2H}}^{t_j^{2H}} \int_0^{t_{j-1}^{2H}}(t_{j-1}^{2H}+y)^{K-1}\d x \d y\\
&-C\int_{t_{j-1}^{2H}}^{t_j^{2H}} y^{K-1} \d y
\end{split}
\end{equation*}
Now for $K>1$ we have $y^{K-1} \leq C$ which leads to
$$
\int_{t_{j-1}^{2H}}^{t_j^{2H}} y^{K-1} \d y \leq C(t_j^{2H} - t_{j-1}^{2H}) \leq C|\pi_n|^{\min(1,2H)}
$$
by the fact that for $T\geq a>b\geq 0$ and $\gamma\in(0,1)$ we have $a^{\gamma}-b^{\gamma} \leq (a-b)^{\gamma}$ and for $\gamma\geq 1$ we have $a^{\gamma}-b^{\gamma} \leq C(a-b)$ by the mean value theorem. Similarly, for $K<1$ we have
$$
\int_{t_{j-1}^{2H}}^{t_j^{2H}} y^{K-1} \d y \leq t_{j}^{2HK} - t_{j-1}^{2HK} \leq |\pi_n|^{\min(1,2HK)}.
$$
Treating other integrals similarly the result follows by Theorem \ref{thm:1st_order_fbm} with $\phi(x) =x^{2HK-1}$.
\end{proof}
\begin{rem}{\rm 
It may look like that for case $K>1$ one gets better (i.e. larger exponent) by computing 
$$
\int_{t_{j-1}^{2H}}^{t_j^{2H}} y^{K-1} \d y = C(t_j^{2HK} - t_{j-1}^{2HK}) \leq C|\pi_n|^{\min(1,2HK)}.
$$
However, this analysis cannot be used to cover, e.g. integral $\int_{t_{j-1}^{2H}}^{t_j^{2H}} (T+y)^{K-1} \d y$.
}\end{rem}
\begin{rem}{\rm 
To compare our results with the existing literature, in \cite{malukas} it was proved that almost sure convergence holds, in our notation, for value $\gamma=\frac{1}{1-HK}$ for the whole range $H\in(0,1)$ and $K\in(0,1]$. Note that by putting $K=1$ and $H=\frac12$ we have a standard Brownian motion, and this produces half of the best possible rate. Now in our result we have value $\gamma =\frac{1}{2-2HK}$ which is twice better compared to $\frac{1}{1-HK}$. Furthermore, we obtained even better for the range $2HK>1$. Note also that, to the best of our knowledge, the case $K>1$ is not studied in the literature before the present paper.
}\end{rem}
\begin{rem}{\rm 
Particularly interesting case is $HK=\frac12$. In this case the quadratic variation exists in the ordinary sense which allows one to develop stochastic calculus with respect to $B^{H,K}$ although $B^{H,K}$ is not a semimartingale \cite{ru-tu} if $K\in(0,1)$. Now for this process we obtain similar condition to the one of standard Brownian motion. On the other hand, if $K\in(1,2)$ and $HK=\frac12$, then the process $B^{H,K}$ is a semimartingale. However, in this case we only obtain condition $|\pi_n|^{2H}=o\left(\frac{1}{\log n}\right)$ since $K>1$ and $HK=\frac12$ implies $H<\frac12$.
}\end{rem}
\begin{rem}{\rm 
In the case $K\in(0,1]$ we obtain a sufficient condition $HK<\frac34$ for the central limit theorem to hold which of course is not at all surprising. Similarly, in the case $K\in(1,2)$ we obtain sufficient condition $HK<\frac34$ provided that $H>\frac12$. However, in the case $H<\frac12$ something odd seems to happen. Indeed, if $HK\geq\frac34$, then $2H+\frac12 - 2HK \leq 0$ so that the given Berry-Esseen bound does not converge to zero. On the other hand, even if $HK<\frac34$ it is not necessarily true that $2H+\frac12 - 2HK>0$ so that condition $HK<\frac34$ is no longer sufficient. Indeed, even in the semimartingale case $2HK=1$ we have $2H+\frac12 -2HK\leq 0$ for values $H\in\left(0,\frac14\right]$. 
}\end{rem}

\textbf{Acknowledgements}
Lauri Viitasaari was partially funded by Emil Aaltonen Foundation.

\appendix
\section{Proof of Lemma \ref{lma:2nd_4th}}
A simple application of \eqref{eq:basic_cumulant} yields
$$
\E V_n^2 = \sum_{k,j=1}^n \left(\E\left[Y^{(n)}_kY^{(n)}_j\right]\right)^2.
$$
Next we compute $\E V_n^4$. We have
\begin{equation}
\label{eq:product}
V_n^4 = \sum_{i,j,k,l=1}^n \prod_{p\in\{i,j,k,l\}}\left[\left(Y^{(n)}_p\right)^2 - \E\left(Y^{(n)}_p\right)^2\right].
\end{equation}
Recall next that all information of a Gaussian vector is encoded to the covariance matrix $\Gamma^{(n)}$ so that $k$-moments of a centred Gaussian vector $(Y_1,Y_2,\ldots,Y_n)$ can be computed via formula
$$
\E[Y_1^{k_1}Y_2^{k_2}\ldots Y_n^{k_n}] = \sum_{\sigma}\E[Y_{\sigma(1)}Y_{\sigma(2)}]\ldots \E[Y_{\sigma(n-1)}Y_{\sigma(n)}],
$$
where the summation is over all permutations $\sigma$ of numbers $\{1,2,\ldots,n\}$, hence producing $n!$ terms. Applying this to vector 8-dimensional vector $\left(Y_k^{(n)},Y_k^{(n)},Y_j^{(n)},\ldots, Y_l^{(n)}\right)$ and taking expectation on 
$$
\left[\left(Y^{(n)}_k\right)^2\left(Y^{(n)}_i\right)^2\left(Y^{(n)}_j\right)^2\left(Y^{(n)}_l\right)^2\right]
$$
we obtain terms of form
$$
A_1(\sigma) = \left[\E\left[Y^{(n)}_{\sigma(1)}Y^{(n)}_{\sigma(2)}\right]\right]^2\left[\E\left[Y^{(n)}_{\sigma(3)}Y^{(n)}_{\sigma(4)}\right]\right]^2,
$$
$$
A_2(\sigma) = \E\left[Y^{(n)}_{\sigma(1)}\right]^2\E\left[Y^{(n)}_{\sigma(2)}\right]^2\left[\E\left[Y^{(n)}_{\sigma(3)}Y^{(n)}_{\sigma(4)}\right]\right]^2,
$$
$$
A_3(\sigma) = \E\left[Y^{(n)}_{\sigma(1)}\right]^2\E\left[Y^{(n)}_{\sigma(2)}\right]^2\E\left[Y^{(n)}_{\sigma(3)}\right]^2\E\left[Y^{(n)}_{\sigma(4)}\right]^2,
$$
$$
A_4(\sigma) = \E\left[Y^{(n)}_{\sigma(1)}Y^{(n)}_{\sigma(2)}\right]\E\left[Y^{(n)}_{\sigma(2)}Y^{(n)}_{\sigma(3)}\right]\E\left[Y^{(n)}_{\sigma(3)}Y^{(n)}_{\sigma(1)}\right]\E\left[Y^{(n)}_{\sigma(4)}\right]^2,
$$
and
$$
A_5(\sigma) = \E\left[Y^{(n)}_{\sigma(1)}Y^{(n)}_{\sigma(2)}\right]\E\left[Y^{(n)}_{\sigma(2)}Y^{(n)}_{\sigma(3)}\right]\E\left[Y^{(n)}_{\sigma(3)}Y^{(n)}_{\sigma(4)}\right]\E\left[Y^{(n)}_{\sigma(4)}Y^{(n)}_{\sigma(1)}\right],
$$
where $\sigma=(\sigma(1),\sigma(2),\sigma(3),\sigma(4))$ can be any permutation of indices $\{i,j,k,l\}$. Next we note by symmetry of covariance and summing over symmetric set $\{1\leq i,j,k,l\leq n\}$ we obtain that for each $p=1,\ldots, 5$ and any permutation $\sigma$ we have
$
\sum_{i,j,k,l=1}^n A_p(\sigma) = \sum_{i,j,k,l=1}^n A_p(\sigma_0),
$
where $\sigma_0$ is any fixed permutation. For example, we obviously have
$$
\sum_{i,j,k,l=1}^n \left[\E\left[Y^{(n)}_{i}Y^{(n)}_{j}\right]\right]^2\left[\E\left[Y^{(n)}_{k}Y^{(n)}_{l}\right]\right]^2 = \sum_{i,j,k,l=1}^n \left[\E\left[Y^{(n)}_{i}Y^{(n)}_{k}\right]\right]^2\left[\E\left[Y^{(n)}_{j}Y^{(n)}_{l}\right]\right]^2.
$$
Consequently we obtain
$$
\sum_{i,j,k,l=1}^n \E\left[\left(Y^{(n)}_k\right)^2\left(Y^{(n)}_i\right)^2\left(Y^{(n)}_j\right)^2\left(Y^{(n)}_l\right)^2\right]\\
= \sum_{i,j,k,l=1}^n \sum_{p=1}^5 a_p A_p(\sigma_0)
$$
for arbitrary reference permutation $\sigma_0$ and some weights $a = (a_1,\ldots,a_5)$. Note also that the weights $a_p,p=1,\ldots,5$ are independent of indices $i,j,k,l$ and the underlying Gaussian process. Now treating rest of the terms in $\prod_{p\in\{i,j,k,l\}}\left[\left(Y^{(n)}_p\right)^2 - \E\left(Y^{(n)}_p\right)^2\right]$ similarly we conclude that
$$
\E V_n^4 = \sum_{i,j,k,l=1}^n \sum_{p=1}^5 b_p A_p(\sigma_0)
$$
with some weights $b=(b_1,\ldots,b_5)$ independent of $i,j,k,l$ and the underlying Gaussian process. Next we claim that $b=(12,0,0,0,24)$. Of course the weight vector $b$ could be computed exactly via combinatorial arguments but we wish to use a more subtle argument by relying on the classical central limit theorem for a sequence of independent standard normal random variables. We begin by computing the values $b_4$ and $b_5$ which are relatively easy to compute directly. First note that terms $A_5$ are produced only by the term $\E\left[\left(Y^{(n)}_k\right)^2\left(Y^{(n)}_i\right)^2\left(Y^{(n)}_j\right)^2\left(Y^{(n)}_l\right)^2\right]$. Furthermore, terms of form $A_5$ are produced by permutations of indices $\{i,j,k,l\}$ which gives $b_5=4! = 24$. Consider next $b_4$. Terms of form $A_4$ are produced from $\E\left[\left(Y^{(n)}_k\right)^2\left(Y^{(n)}_i\right)^2\left(Y^{(n)}_j\right)^2\left(Y^{(n)}_l\right)^2\right]$ by first picking one variable, $Y^{(n)}_k$ say, to get $\E\left[Y^{(n)}_k\right]^2$ and then organising the remaining three into $3!=6$ ways which produces $4!=24$ (the first one can be picked in 4 ways). On the other hand, computing the product \eqref{eq:product} we obtain 4 terms of form $\E\left[\left(Y^{(n)}_k\right)^2\right]\left(Y^{(n)}_i\right)^2\left(Y^{(n)}_j\right)^2\left(Y^{(n)}_l\right)^2$ and with similar analysis we obtain that each term produces $A_4$ exactly $3!=6$ times. Due to the minus sign in terms $-\E\left[\left(Y^{(n)}_k\right)^2\right]$ and the fact $24-4\times 6 = 0$ we obtain $b_4=0$. It remains to prove that $b_1 = 12$ and $b_2=b_3 = 0$. For this purpose let $Y^{(n)}_k$ be a sequence of independent standard normal random variables $Y_k$. Then by the classical central limit theorem we have
$$
S_n := \frac{1}{\sqrt{2n}}\sum_{k=1}^n [Y_k^2 - \E Y_k^2] \to \mathcal{N}(0,1)
$$
in distribution and consequently, $\E S_n^4 \to 3$. In this case we have
$$
b_1 \sum_{i,j,k,l}^n A_1(\sigma_0) = b_1n^2
$$
$$
b_2 \sum_{i,j,k,l=1}^n  A_3(\sigma_0) = b_2n^3
$$
and 
$$
b_4 \sum_{i,j,k,l=1}^n  A_2(\sigma_0) = b_3n^4
$$
so that
$$
\E S_n^4 = \frac{1}{4n^2}\left[b_1 n^2 + b_2n^3 + b_3n^4\right].
$$
Now since $b_k,k=1,2,3$ is independent of $n$ and the underlying Gaussian process, the convergence $\E S_n^4 \to 3$ implies $b_2=b_3=0$ and $b_1 = 12$. This completes the proof.

\bibliographystyle{plain}
\bibliography{bibli}
\end{document}